\documentclass[a4paper]{article}
\usepackage{geometry}
\usepackage[tbtags]{amsmath}
\usepackage{amsthm}
\usepackage{amssymb}
\usepackage{cite}
\usepackage{booktabs}
\usepackage[pdftex,colorlinks]{hyperref}
\usepackage{subfigure}
\usepackage{graphicx}
\usepackage{epstopdf}
\usepackage{epsf}
\usepackage{tikz}
\tikzset{>=stealth}
\usepackage{multicol}
\usepackage{multirow}
\usepackage{bm}
\usepackage[width=.7\textwidth]{caption}

\numberwithin{equation}{section}

\newcommand{\fanshu}[2][0]{\Vert#2\Vert_{#1}}
\newcommand{\banfan}[2][0]{|#2|_{#1}}
\newcommand{\mesh}{\mathcal{T}_h}
\newcommand{\face}{\mathcal{F}_h}
\newcommand{\curl}{\nabla\times}
\newcommand{\curlc}{(\nabla\times)^2}
\newcommand{\gcurl}{\nabla\nabla\times}
\newcommand{\divv}{\nabla\cdot}

\newcommand{\jump}[2][]{[\![#2]\!]_{#1}}
\newcommand{\norma}[1]{\vert\!\vert\!\vert#1\vert\!\vert\!\vert_{h}}

\newtheorem{theorem}{\noindent{\bf Theorem}}[section]

\newtheorem{lemma}{\noindent{Lemma}}[section]

\theoremstyle{remark}
\newtheorem{remark}{\noindent{Remark}}[section]

\begin{document}\large 
\title{{\large  \textbf{A new family of nonconforming elements with $\pmb{H}(\mathrm{curl})$-continuity for the three-dimensional quad-curl problem}}}
\author{\small{Baiju Zhang}\thanks{Beijing Computational Science Research Center, Beijing 100193, China, \tt baijuzhang@csrc.ac.cn} \and \small{Zhimin Zhang}\thanks{Beijing Computational Science Research Center, Beijing 100193, China, {\tt zmzhang@csrc.ac.cn}; Department of Mathematics, Wayne State University, Detroit, MI 48202, USA, \tt ag7761@wayne.edu}}
\date{}\maketitle 

\begin{abstract}
We propose and analyze a new family of nonconforming finite elements for the three-dimensional quad-curl problem. The proposed finite element spaces are subspaces of $\pmb{H}(\mathrm{curl})$, but not of $\pmb{H}(\mathrm{grad}~\mathrm{curl})$, which are different from the existing nonconforming ones \cite{Hong2012DGFourthOrderCurl,huang2022nonconforing,hu2022afamily}. The well-posedness of the discrete problem is proved and optimal error estimates in discrete $\pmb{H}(\mathrm{grad}~\mathrm{curl})$ norm, $\pmb{H}(\mathrm{curl})$ norm and $\pmb{L}^2$ norm are derived. Numerical experiments are provided to illustrate the good performance of the method and confirm our theoretical predictions.
\end{abstract}

\section{Introduction}
The quad-curl problem arises in many areas such as inverse electromagnetic scattering \cite{cakoni2007variational,monk2012finite,sun2011iterative} and magnetohydrodynamics \cite{zheng2011nonconforming}. Finite element methods (FEMs) are natural choices for numerical treatment of such problems. In recent years, various FEMs have been proposed and analyzed.

Among various FEMs, conforming elements are  natural candidate, for their error estimates can be obtained by the standard framework. Conforming methods on triangles and rectangles in two dimensions were devised in \cite{zhang2019Hcurl2,hu2020Simple}. Then a family of curl-curl conforming elements on tetrahedra was constructed in \cite{zhang2020curlcurl}, which has at least 315 degrees of freedom (DOFs) on each tetrahedron. Later, by enriching the shape function space with macro-element bubble functions, the authors of \cite{hu2022afamily} constructed a family of conforming elements whose lowest order element has only 18 DOFs. However, the use of macro-element bubble functions adds some difficulty in coding compared with the finite element space using only pure polynomials.

To avoid the use of conforming finite elements, nonconforming FEMs were studied in \cite{zheng2011nonconforming,huang2022nonconforing}. These methods have low computational cost since they have
small number of degrees of freedom (DOFs), but both of them are low-order. The authors of \cite{hu2022afamily} proposed another $\pmb{H}(\mathrm{grad}~\mathrm{curl})$-nonconforming elements, but macro-element bubble functions are still required.

Another approach to avoid the use of conforming finite elements is to write the quad-curl problem as a system of second-order problems and use the mixed FEM \cite{Sun2016MixedFemQuad}. Different mixed schemes for the quad-curl problems were presented in \cite{Zhang2018MixedQuadCurl}. Instead of solving the quad-curl problem directly, finite element methods based on decoupling of mixed formulation were studied in \cite{Brenner2017Hodge,cao2022error}.

In addition to the above methods, discontinuous Galerkin (DG) methods are also used to deal with the operator $(\nabla\times)^4$. In \cite{Hong2012DGFourthOrderCurl}, a DG method based on $H(\mathrm{curl})$-conforming element was presented for the quad-curl problem. A hybridizable discontinuous Galerkin (HDG) method was investigated in \cite{Chen2019HDGQuadCurl}. Recently, a $C^0$ interior penalty method was studied in \cite{Sun2020C0interior}. However, these methods modify the variational formulation to tackle discontinuity of the basis functions, which increases the difficulty of programming to a certain extent.

The main purpose of this work is to construct a family of nonconforming elements which uses only pure polynomials as shape functions, allows high order extension, and is easier to code than conforming, DG and HDG methods. By comparison, we find that the nonconforming space $\pmb{U}_h$ satisfies $\pmb{U}_h\nsubseteq\pmb{H}(\mathrm{curl}),\  \nabla_h\times\pmb{U}_h\nsubseteq\pmb{H}^1$ in  \cite{zheng2011nonconforming,huang2022nonconforing}, and satisfies $\pmb{U}_h\nsubseteq\pmb{H}(\mathrm{curl}),\  \nabla_h\times\pmb{U}_h\subset\pmb{H}^1$ in  \cite{hu2022afamily}. A natural question is whether there is a space that satisfies $\pmb{U}_h\subset\pmb{H}(\mathrm{curl}),\  \nabla\times\pmb{U}_h\nsubseteq\pmb{H}^1$,
since for $\pmb{u}_h$ in such a space, we would have $\curl\pmb{u}_h\in\pmb{H}(\mathrm{div})$, $\curl\pmb{u}_h$ is divergence-free and has some tangential continuity. This reminds us nonconforming elements for the Stokes and Brinkman problems \cite{guzman2012afamily,xie2008uniformly}. The main idea to construct nonconforming elements in \cite{guzman2012afamily} is to modify $\pmb{H}(\mathrm{div})$-conforming finite elements to have some tangential continuity. More precisely, their local space on each simplex $K$ is of the form
\begin{align*}
\pmb{M}(K)+\curl(b_K\pmb{Q}(K)),
\end{align*}
where $\pmb{M}(K)$ is the local $\pmb{H}(\mathrm{div})$-conforming space, $b_K$ is the element bubble function that vanishes on $\partial K$ and the space $\pmb{Q}(K)$ is spanned by the face bubble functions multiplied by some vector-valued polynomials. Inspired by this, we construct our local space in the form
\begin{align*}
\pmb{N}(K)+b_K\pmb{Q}(K),
\end{align*}
where $\pmb{N}(K)$ is the local $\pmb{H}(\mathrm{curl})$-conforming space. Note that $b_K$ vanishes on $\partial K$ and hence the resulting space is still $\pmb{H}(\mathrm{curl})$-conforming. Thus, the only purpose of adding the function $b_K\pmb{Q}(K)$ is to enforce some tangential continuity of $\curl\pmb{u}_h$. As $\pmb{Q}(K)$  in \cite{guzman2012afamily} being arbitrary order, the nonconforming space constructed here can also be arbitrary order. Moreover, the proposed method only uses pure polynomials as shape functions and the curl of the finite element space possesses some tangential continuity, so the implement of the method is relatively simpler than conforming, DG and HDG methods.

The rest of this paper is organized as follows: In Section \ref{sectionpreliminary}, we introduce some notations. In Section \ref{sectionFEM}, we introduce a family of nonconforming elements for the quad-curl problem in 3D. Here,
we describe the space, its associated degrees of freedom and unisolvency. We also define the related interpolation and study its approximation properties. In Section \ref{sectionApplication}, we use our proposed elements to solve the quad-curl problem and give error estimates. In Section \ref{sectionnumerical}, we provide numerical examples to verify the correctness and efficiency of our method.

Throughout the paper, we use the notations $A\lesssim B$ and $A\gtrsim B$ to represent the inequalities $A\le CB$ and $A\ge CB$, respectively, where $C$ may depend only on the minimum angle of the triangles in the mesh $\mesh$.

\section{Preliminaries}\label{sectionpreliminary}
Let $\Omega\subset\mathbb{R}^3$ be a contractible Lipschitz domain and $\partial \Omega$ its boundary. For a positive integer, we utilize common notation for the Sobolev spaces $H^m(D)$ or $H^m_0(D)$ on a simply connected subdomain $D\subset\Omega$, equipped with the norm $\fanshu[m,D]{\cdot}$ and seminorm $\banfan[m,D]{\cdot}$. If $D=\Omega$, the subscript will be omitted. Conventionally, we write $L^2(D)$ instead of $H^0(D)$.

Let $\pmb{u}=(u_1,u_2,u_3)^T$ and $\pmb{w}=(w_1,w_2,w_3)^T$, where the superscript $T$ denotes the transpose, then we denote
\begin{align*}
\divv\pmb{u}&=\frac{\partial u_1}{\partial x_1}+\frac{\partial u_2}{\partial x_2}+\frac{\partial u_3}{\partial x_3},\\
\pmb{u}\times\pmb{w}&=(u_2w_3-u_3w_2,u_1w_3-u_3w_1,u_1w_2-u_2w_1)^T,\\
\curl\pmb{u}&=(\frac{\partial u_3}{\partial x_2}-\frac{\partial u_2}{\partial x_3},\frac{\partial u_1}{\partial x_3}-\frac{\partial u_3}{\partial x_1},\frac{\partial u_2}{\partial x_1}-\frac{\partial u_1}{\partial x_2})^T,\\
(\curl)^2\pmb{u}&= \curl\curl\pmb{u}.
\end{align*}

On a simply-connected sub-domain $D\subset\Omega$ we write
\begin{align}
L_0^2({D})&=\{q\in L^2({D}): \int_{{D}}qd\pmb{x}=0\},\\
\pmb{H}(\mathrm{div};{D})&=\{\pmb{v}\in \pmb{L}^2({D}): \divv \pmb{v}\in L^2({D})\},\\
\pmb{H}(\mathrm{div}^0;{D})&=\{\pmb{v}\in \pmb{H}(\mathrm{div};{D}):\divv\pmb{v}=0 \text{ in }{D}\},\\
\pmb{H}(\mathrm{curl};{D})&=\{\pmb{v}\in \pmb{L}^2({D}): \curl \pmb{v}\in \pmb{L}^2({D})\},\\
\pmb{H}_0(\mathrm{curl};{D})&=\{\pmb{v}\in \pmb{H}(\mathrm{curl};{D}):\pmb{v}\times\pmb{n}=0 \text{ on }\partial{D}\},\\
\pmb{H}^1(\mathrm{curl};{D})&=\{\pmb{v}\in \pmb{H}^1({D}): \curl\pmb{v}\in \pmb{H}^1({D})\},\\
\pmb{H}^1_0(\mathrm{curl};{D})&=\{\pmb{v}\in \pmb{H}^1(\mathrm{curl};{D})\cap\pmb{H}_0^1({D}):\ \curl\pmb{v}=0 \text{ on }\partial{D}\},\\
\pmb{H}(\mathrm{grad}\ \mathrm{curl};{D})&=\{\pmb{v}\in \pmb{L}^2({D}): \curl\pmb{v}\in\pmb{H}^1({D})\},\\
\pmb{H}_0(\mathrm{grad}\ \mathrm{curl};{D})&=\{\pmb{v}\in\pmb{H}(\mathrm{grad}\ \mathrm{curl};{D}):\pmb{v}\times\pmb{n}=0, \ \curl\pmb{v}=0 \text{ on }\partial{D}\},
\end{align}
where $\pmb{n}$ is the unit outward normal vector to the boundary $\partial D$.

Let $\mesh$ be a shape-regular simplicial triangulation with $h_K=\mathrm{diam}(K)$ for all $K\in\mesh$ and $h=\max_{K\in\mesh}h_K$. We denote by $\face$ the faces, $\face^i$ the interior faces and $\face^b$ the boundary faces in $\mesh$. Given $K\in\mesh$ and a face $F\subset\partial K$, we denote by $\lambda_F$ the barycentric coordinate of $K$ which vanishes on $F$. The element bubble and face bubbles are given by
\begin{align}
b_K=\prod_F\lambda_F,\quad b_F=\prod_{G\neq F}\lambda_G,
\end{align}
where the product runs over the faces of $K$.

For a given simplex $S$ in $\mathbb{R}^3$ and $m\ge0$, the vector-valued polynomials are defined as $\pmb{P}_m(S)=[P_m(S)]^3$, where $P_m(S)$ is the space of polynomials defined of $S$ of degree less than or equal to $m$. We also set $\pmb{P}_m(S)$ and $P_m(S)$ to be the empty set if $m<0$.

For interior and boundary inner products we will use the following notation:
\begin{align*}
(\pmb{u},\pmb{v})_K=\int_K\pmb{u}\cdot\pmb{v}dx,\quad \langle p,q\rangle_F=\int_F pqds.
\end{align*}
We denote by $\pmb{n}_F$ the unit outward normal vector to a face $F$ of $K$.

We will also need to define the tangential and normal jump operators. If $F\in \face^i$ with $F=K^+\cap K^-$, then we set
\begin{align*}
\jump{\pmb{v}\times\pmb{n}}|_F&=(\pmb{v}|_{K^+}\times\pmb{n}_{K^+})|_F+(\pmb{v}|_{K^-}\times\pmb{n}_{K^-})|_F,\\
\jump{\pmb{v}\cdot\pmb{t}}|_F&=(\pmb{v}|_{K^+}\cdot\pmb{t}_{K^+})|_F+(\pmb{v}|_{K^-}\cdot\pmb{t}_{K^-})|_F,\\
\jump{\pmb{v}\cdot\pmb{n}}|_F&=(\pmb{v}|_{K^+}\cdot\pmb{n}_{K^+})|_F+(\pmb{v}|_{K^-}\cdot\pmb{n}_{K^-})|_F,
\end{align*}
where $\pmb{n}_{K^{\pm}}$ is the unit outward normal to $\partial K^{\pm}$ and $\pmb{t}_{K^{\pm}}$ is the unit outward tangent to $\partial K^{\pm}$.

If $F\in \face^b$ is a boundary face with $F=K^+\cap K^-$, then we set
\begin{align*}
\jump{\pmb{v}\times\pmb{n}}|_F&=(\pmb{v}|_{K}\times\pmb{n}_{K})|_F,\\
\jump{\pmb{v}\cdot\pmb{t}}|_F&=(\pmb{v}|_{K}\cdot\pmb{t}_{K})|_F,\\
\jump{\pmb{v}\cdot\pmb{n}}|_F&=(\pmb{v}|_{K}\cdot\pmb{n}_{K})|_F.
\end{align*}

\section{A family of nonconforming curl-curl finite element in 3D}\label{sectionFEM}
\subsection{The local space}
Since our new elements are based on $\pmb{H}(\mathrm{curl};\Omega)$ finite element spaces plus some bubble functions, we first review some well-known elements. For $K\in\mesh$, let $\pmb{N}^k(K)$ denote the local N\'{e}d\'{e}lec space of the first kind,
\begin{align}
\pmb{N}^k(K)=\pmb{P}^{k}(K)+\{\pmb{v}\in\pmb{P}^{k+1}(K):\pmb{v}\cdot\pmb{x}=0\}\quad(k\ge1),\label{Nedelec1}
\end{align}
and let $\pmb{M}^k(K)$ denote the local BDM space
\begin{align}
\pmb{M}^{k}(K)=\pmb{P}^{k}(K)\quad(k\ge1).\label{BDM}
\end{align}

We then define the local space for the 3D quad-curl problem as
\begin{align}
\pmb{U}^k(K)=\pmb{N}^k(K)+b_K\pmb{Q}^{k-1}(K),
\end{align}
where
\begin{align}
\pmb{Q}^{k-1}(K)=\sum_{F}b_F\pmb{Q}_F^{k-1}(K),
\end{align}
and
\begin{align}
\pmb{Q}_F^{k-1}(K)=\{&\pmb{q}\times\pmb{n}_F\in\pmb{P}^{k-1}(K)\times\pmb{n}_F:\notag\\
&\quad(\pmb{q}\times\pmb{n}_F,b_Kb_F(\pmb{w}\times\pmb{n}_F))_K=0\quad\forall\pmb{w}\in\pmb{P}^{k-2}(K) \}.\label{definitionQF}
\end{align}
The space $b_K\pmb{Q}^{k-1}(K)$ was introduced by \cite{guzman2012afamily}.

We equip the local space $\pmb{U}^k(K)$  with the following degrees of freedom
\begin{align}
&\langle\pmb{u}\cdot\pmb{t}_e,\kappa\rangle_e\qquad\qquad\qquad\text{for all }\kappa\in P^{k}(e) \text{ and edge }e\text{ of }K,\label{nonconDOFs1}\\
&\langle\pmb{u}\times\pmb{n}_F,\pmb{\mu}\rangle_F\qquad\qquad\ \text{ for all }\pmb{\mu}\in \pmb{P}^{k-1}(F) \text{ and faces }F\text{ of }K,\label{nonconDOFs2}\\
&\langle(\curl\pmb{u})\times\pmb{n}_F,\pmb{\chi}\rangle_F\quad\ \text{ for all }\pmb{\chi}\in \pmb{P}^{k-1}(F) \text{ and faces }F\text{ of }K,\label{nonconDOFs3}\\
&(\pmb{u},\pmb{\rho})_K\qquad\qquad\qquad\quad\text{for all }\pmb{\rho}\in \pmb{P}^{k-2}(K).\label{nonconDOFs4}
\end{align}

\begin{theorem}
We have
\begin{align}
\pmb{U}^k(K)&=\pmb{N}^k(K)\oplus b_K\pmb{Q}^{k-1}(K),\label{UKNKbKQF}\\
\dim \pmb{U}^k(K)&= \dim\pmb{N}^k(K)+ 4\dim \pmb{P}^{k-1}(F).\label{dimUK}
\end{align}
Furthermore, any function $\pmb{v}\in \pmb{U}^k(K)$ is uniquely determined by the degrees of freedom \eqref{nonconDOFs1}-\eqref{nonconDOFs4}.
\end{theorem}
\begin{proof}
Let $\pmb{z}\in b_K\pmb{Q}^{k-1}(K)$ with $\pmb{z}=b_K\sum_Fb_F(\pmb{q}_F\times\pmb{n}_F)$ and $\pmb{q}_F\times\pmb{n}_F\in\pmb{Q}_F^{k-1}$. For any $\pmb{\rho}\in\pmb{P}^{k-2}(K)$, we have $\pmb{\rho}=\pmb{\rho}\cdot\pmb{n}_F-(\pmb{\rho}\times\pmb{n}_F)\times\pmb{n}_F.$ By the definition of $\pmb{Q}_F^{k-1}$ \eqref{definitionQF}, we derive
\begin{align}\label{zrho}
(\pmb{z},\pmb{\rho})_K&=\sum_F(b_Kb_Fq_F\times\pmb{n}_F,\pmb{\rho})_K\notag\\
&=\sum_F(q_F\times\pmb{n}_F,b_Kb_F,(-\pmb{\rho}\times\pmb{n}_F)\times\pmb{n}_F)_K=0,\ \forall \pmb{\rho}\in\pmb{P}^{k-2}(K).
\end{align}
With this in hand, if $\pmb{v}\in\pmb{N}^k(K)\cap b_K\pmb{Q}^{k-1}(K)$, then we obtain
\begin{align*}
&\langle\pmb{v}\cdot\pmb{t}_e,\kappa\rangle_e=0\qquad\qquad\qquad\text{for all }\kappa\in P^{k}(e) \text{ and edge }e\text{ of }K,\\
&\langle\pmb{v}\times\pmb{n}_F,\pmb{\mu}\rangle_F=0\qquad\qquad\ \text{ for all }\pmb{\mu}\in \pmb{P}^{k-1}(F) \text{ and faces }F\text{ of }K,\\
&(\pmb{v},\pmb{\rho})_K=0\qquad\qquad\qquad\quad\text{for all }\pmb{\rho}\in \pmb{P}^{k-2}(K),
\end{align*}
which means that all the degrees of freedom of $\pmb{v}\in\pmb{N}^k(K)$ vanish. Thus we get $\pmb{v}=0$. Therefore we have proved \eqref{UKNKbKQF}. \eqref{dimUK} is a direct result of \eqref{UKNKbKQF} and
\cite[Lemma 3.2]{guzman2012afamily}.

It is easy to check that $\dim\pmb{U}^k(K)$ is equal to the number of degrees of freedom given by \eqref{nonconDOFs1}-\eqref{nonconDOFs4}. As a result, it suffices to show that \eqref{nonconDOFs1}-\eqref{nonconDOFs4} vanish for $\pmb{v}\in\pmb{U}^k(K)$, then $\pmb{v}=0$.

Let $\pmb{v}=\pmb{v}_0+\pmb{z}$ where $\pmb{v}_0\in\pmb{N}^k(K)$ and $\pmb{z}\in b_K\pmb{Q}^{k-1}(K).$ Recalling \eqref{zrho}, we have
\begin{align*}
&\langle\pmb{v}_0\cdot\pmb{t}_e,\kappa\rangle_e=0\qquad\qquad\qquad\text{for all }\kappa\in P^{k}(e) \text{ and edge }e\text{ of }K,\\
&\langle\pmb{v}_0\times\pmb{n}_F,\pmb{\mu}\rangle_F=0\qquad\qquad\ \text{ for all }\pmb{\mu}\in \pmb{P}^{k-1}(F) \text{ and faces }F\text{ of }K,\\
&(\pmb{v}_0,\pmb{\rho})_K=0\qquad\qquad\qquad\quad\text{for all }\pmb{\rho}\in \pmb{P}^{k-2}(K),
\end{align*}
which implies $\pmb{v}_0=0$. Consequently, we derive
\begin{align}
\pmb{v}=\pmb{z}=b_K\sum_Fb_F(\pmb{q}_F\times\pmb{n}_F),
\end{align}
for some $\pmb{q}_F\times\pmb{n}_F\in\pmb{Q}_F^{k-1}$.

Since the degrees of freedom \eqref{nonconDOFs3}, we get from \cite[Lemma 3.1]{guzman2012afamily} that
\begin{align}
0=\langle(\curl\pmb{z})\times\pmb{n}_F,\pmb{q}_F\times\pmb{n}_F\rangle_F=-\langle b_F^2(\pmb{q}_F\times\pmb{n}_F),\pmb{q}_F\times\pmb{n}_F\rangle_F,
\end{align}
which means that $\pmb{q}_F\times\pmb{n}_F$ vanishes on $F$. Then by \cite[Lemma 3.2]{guzman2012afamily} we have that $\pmb{q}_F\times\pmb{n}_F=0$ on $K$. We have thus proved the theorem.
\end{proof}
In addition, we introduce the local space\cite{guzman2012afamily} for the 3D Brinkman problem as
\begin{align}
\pmb{V}^k(K)=\pmb{M}^k(K)+\curl(b_K\pmb{Q}^{k-1}(K)).
\end{align}
The following degrees of freedom define a function $\pmb{v}\in\pmb{V}^k(K)$:
\begin{align}
&\langle\pmb{v}\cdot\pmb{n}_F,\mu\rangle_F\qquad\qquad\ \text{ for all }\mu\in P^{k}(F) \text{ and faces }F\text{ of }K,\label{GNDOFs1}\\
&\langle\pmb{v}\times\pmb{n}_F,\pmb{\chi}\rangle_F\qquad\qquad\ \text{ for all }\pmb{\chi}\in \pmb{P}^{k-1}(F) \text{ and faces }F\text{ of }K,\label{GNDOFs2}\\
&(\pmb{v},\pmb{\rho})_K\qquad\qquad\qquad\quad\text{for all }\pmb{\rho}\in \pmb{N}^{k-2}(K).\label{GNDOFs3}
\end{align}
\subsection{The global space and interpolation}
We have defined the local finite element spaces, then we can define the associated global spaces as
\begin{align}
\pmb{U}_h = \{\pmb{u}\in& \pmb{H}(\mathrm{curl};\Omega):\ \pmb{u}|_K\in\pmb{U}^k(K),\ \forall K\in\mesh,\notag\\
&\langle\jump{(\curl\pmb{u})\times\pmb{n}},\pmb{\chi}\rangle_F=0\ \forall\pmb{\chi}\in\pmb{P}^{k-1}(F),\ \forall F\in\face^i\},\label{definitionUh}\\
\pmb{U}_{h,0} = \{\pmb{u}\in& \pmb{U}_h\cap\pmb{H}_0(\mathrm{curl};\Omega): \langle\jump{(\curl\pmb{u})\times\pmb{n}},\pmb{\chi}\rangle_F=0\ \forall\pmb{\chi}\in\pmb{P}^{k-1}(F),\ \forall F\in\face^b\},\label{definitionUh0}\\
\pmb{B}_h = \{\pmb{u}\in& \pmb{H}(\mathrm{curl};\Omega):\ \pmb{u}|_K\in b_K\pmb{Q}^{k-1}(K),\ \forall K\in\mesh\},\label{definitionBh}\\
\pmb{V}_h = \{\pmb{v}\in& \pmb{H}(\mathrm{div};\Omega):\ \pmb{v}|_K\in\pmb{V}^k(K),\ \forall K\in\mesh,\notag\\
&\langle\jump{\pmb{v}\times\pmb{n}},\pmb{\chi}\rangle_F=0\ \forall\pmb{\chi}\in\pmb{P}^{k-1}(F),\ \forall F\in\face^i\},\label{definitionVh}\\
\pmb{V}_{h,0} = \{\pmb{v}\in& \pmb{V}_h\cap\pmb{H}_0(\mathrm{div};\Omega): \langle\jump{\pmb{v}\times\pmb{n}},\pmb{\chi}\rangle_F=0\ \forall\pmb{\chi}\in\pmb{P}^{k-1}(F),\ \forall F\in\face^b\},\label{definitionVh0}\\
\pmb{N}_h = \{\pmb{u}\in& \pmb{H}(\mathrm{curl};\Omega):\ \pmb{u}|_K\in\pmb{N}^k(K),\ \forall K\in\mesh\},\\
\pmb{M}_h = \{\pmb{v}\in& \pmb{H}(\mathrm{div};\Omega):\ \pmb{v}|_K\in\pmb{M}^k(K),\ \forall K\in\mesh\},\\
W_{h,0} = \{w\in& H_0^1(\Omega):\ w|_K\in P^{k+1}(K),\ \forall K\in\mesh\}.
\end{align}

Using the degrees of freedom \eqref{nonconDOFs1}-\eqref{nonconDOFs4}, we can define the interpolation $\pmb{\Pi}_{U}:\pmb{H}^1(\mathrm{curl};\Omega)\rightarrow\pmb{U}_h$ given locally as follows:
\begin{align}
&\langle(\pmb{\Pi}_{U}\pmb{u}-\pmb{u})\cdot\pmb{t}_e,\kappa\rangle_e=0\qquad\qquad\qquad\text{for all }\kappa\in P^{k}(e) \text{ and edge }e\text{ of }K,\label{nonconInterp1}\\
&\langle(\pmb{\Pi}_{U}\pmb{u}-\pmb{u})\times\pmb{n}_F,\pmb{\mu}\rangle_F=0\qquad\qquad\ \text{ for all }\pmb{\mu}\in \pmb{P}^{k-1}(F) \text{ and faces }F\text{ of }K,\label{nonconInterp2}\\
&\langle(\curl(\pmb{\Pi}_{U}\pmb{u}-\pmb{u}))\times\pmb{n}_F,\pmb{\chi}\rangle_F=0\quad\ \text{ for all }\pmb{\chi}\in \pmb{P}^{k-1}(F) \text{ and faces }F\text{ of }K,\label{nonconInterp3}\\
&(\pmb{\Pi}_{U}\pmb{u}-\pmb{u},\pmb{\rho})_K=0\qquad\qquad\qquad\quad\text{for all }\pmb{\rho}\in \pmb{P}^{k-2}(K).\label{nonconInterp4}
\end{align}
Similarly, we can define the following interpolations
\begin{align*}
&\pmb{\Pi}_{B}:\pmb{H}^1(\mathrm{curl};\mesh)\rightarrow\pmb{B}_h,\quad\pmb{\Pi}_{V}:\pmb{H}^1(\Omega)\rightarrow\pmb{V}_h,\\
&\pmb{\Pi}_{N}:\pmb{H}^1(\Omega)\rightarrow\pmb{N}_h,\quad\pmb{\Pi}_{M}:\pmb{H}^1(\Omega)\rightarrow\pmb{M}_h,
\end{align*}
where $\pmb{H}^1(\mathrm{curl};\mesh)=\{\pmb{u}\in \pmb{L}^2(\Omega):\ \pmb{u}|_K\in\pmb{H}^1(\mathrm{curl};K),\ \forall K\in\mesh\}$.

By a similar argument as in \cite[Lemma 5.40]{monk2003finite}, we can establish the following relation between $\pmb{\Pi}_{U}$ and $\pmb{\Pi}_{V}$.
\begin{lemma}\label{LemmaPiUPiV}
Suppose $\pmb{U}_h$ and $\pmb{V}_h$ are given by \eqref{definitionUh} and \eqref{definitionVh}, respectively. Then $\curl\pmb{U}_h\subset\pmb{V}_h$. If $\pmb{u}\in\pmb{U}$, then
\begin{align}
\curl(\pmb{\Pi}_{U}\pmb{u})=\pmb{\Pi}_{V}\curl\pmb{u}.
\end{align}
\end{lemma}
\begin{lemma}\label{Lemmajumpcurl}
Under the assumption of Lemma \ref{LemmaPiUPiV}, for any $\pmb{v}_h\in\pmb{U}_h,$ we have
\begin{align}
\langle\jump{\curl\pmb{v}_h},\pmb{\chi}\rangle_F=0,\ \forall\pmb{\chi}\in\pmb{P}^{k-1}(F),\ F\in\face^i.
\end{align}
\end{lemma}
\begin{proof}
On each $F\in\face^i$, we can write $\curl\pmb{v}_h$ as
\begin{align}
\curl\pmb{v}_h|_F=\curl\pmb{v}_h\cdot\pmb{n}_F|_F+\pmb{n}_F\times(\curl\pmb{v}_h\times\pmb{n}_F)|_F.
\end{align}
Since $\curl\pmb{U}_h\subset\pmb{V}_h$, we get $\jump{\curl\pmb{v}_h\cdot\pmb{n}_F}|_F=0$. Therefore we obtain
\begin{align}
\langle\jump{\curl\pmb{v}_h},\pmb{\chi}\rangle_F=\langle\jump{(\curl\pmb{v}_h\times\pmb{n}_F)},\pmb{\chi}\times\pmb{n}_F\rangle_F=0, \ \forall\pmb{\chi}\in\pmb{P}^{k-1}(F),\ F\in\face^i.
\end{align}
\end{proof}
We now investigate the approximation properties of $\pmb{\Pi}_{U}$. The next result provides error estimates for the interpolation.
\begin{theorem}\label{theoreminterpo}
Let integers $l$ and $s$ satisfy $1\le l\le s\le k+1$. Then for any $\pmb{u}\in \pmb{H}^{l}(K)$ and $\curl\pmb{u}\in \pmb{H}^{s}(K)$, there holds
\begin{align}
\fanshu[0,K]{\pmb{u}-\pmb{\Pi}_{U}\pmb{u}}&\lesssim h_K^l(|\pmb{u}|_{l,K}+|\curl\pmb{u}|_{s,K}),\label{errorL2PiU}\\
\fanshu[m,K]{\curl\pmb{u}-\curl(\pmb{\Pi}_{U}\pmb{u})}&\lesssim h_K^{s-m}|\curl\pmb{u}|_{s,K}, \;\; m=0,1. \label{errorH1PiU}
\end{align}
\end{theorem}
\begin{proof}
To begin with, \eqref{errorH1PiU} can be easily obtained by Lemma \ref{LemmaPiUPiV} and approximation properties of $\pmb{\Pi}_{V}$ \cite[Theorem 3.4]{guzman2012afamily}.

We now turn to \eqref{errorL2PiU}. It is easy to check that for $\pmb{u}\in\pmb{U}$ we have
\begin{align}
\pmb{u}-\pmb{\Pi}_{U}\pmb{u}=\pmb{u}-\pmb{\Pi}_{N}\pmb{u}-\pmb{\Pi}_{B}(\pmb{u}-\pmb{\Pi}_{N}\pmb{u}).\label{identityPiu}
\end{align}
The above equation suggests us to investigate the stability of $\pmb{\Pi}_{B}$. By the definition of $\pmb{B}_h$ \eqref{definitionBh}, we can write $\pmb{\Pi}_{B}\pmb{u}$ on each $K\in\mesh$ as
\begin{align*}
\pmb{\Pi}_{B}\pmb{u}|_K=b_K\sum_{F}b_F\pmb{q}_F\times\pmb{n}_F\quad\text{with }\pmb{q}_F\times\pmb{n}_F\in\pmb{Q}_F^{k-1}(K).
\end{align*}
According to \cite[Lemma 3.1]{guzman2012afamily}, we have
\begin{align}
(\curl(\pmb{\Pi}_{B}\pmb{u}))\times\pmb{n}_F|_F=-a_Fb_F^2(\pmb{q}_F\times\pmb{n}_F)|_F,\label{equation1}
\end{align}
where $a_F=|\nabla\lambda_F|$. By the above equation and the definition of $\pmb{\Pi}_{B}$, we get
\begin{align*}
a_F\langle b_F^2(\pmb{q}_F\times\pmb{n}_F),\pmb{q}_F\times\pmb{n}_F\rangle_F&=-\langle (\curl(\pmb{\Pi}_{B}\pmb{u}))\times\pmb{n}_F,\pmb{q}_F\times\pmb{n}_F\rangle_F\\
&=-\langle (\curl\pmb{u})\times\pmb{n}_F,\pmb{q}_F\times\pmb{n}_F\rangle_F\\
&\le\fanshu[0,F]{(\curl\pmb{u})\times\pmb{n}_F}\fanshu[0,F]{\pmb{q}_F\times\pmb{n}_F},
\end{align*}
which means that
\begin{align}
a_F\fanshu[0,F]{\pmb{q}_F\times\pmb{n}_F}\lesssim \fanshu[0,F]{(\curl\pmb{u})\times\pmb{n}_F}.\label{inequality1}
\end{align}
Hence, by a scaling argument argument, \eqref{equation1} and \eqref{inequality1}, we have
\begin{align*}
\fanshu[0,K]{\pmb{\Pi}_{B}\pmb{u}}&\lesssim h_K^{3/2}\sum_F\fanshu[0,F]{\curl(\pmb{\Pi}_{B}\pmb{u}\times\pmb{n}_F)}=h_K^{3/2}\sum_F\fanshu[0,F]{a_Fb_F^2(\pmb{q}_F\times\pmb{n}_F)}\notag\\
&\lesssim h_K^{3/2}\sum_Fa_F\fanshu[0,F]{(\pmb{q}_F\times\pmb{n}_F)}\lesssim h_K^{3/2}\fanshu[0,F]{(\curl\pmb{u})\times\pmb{n}_F}.
\end{align*}
Recalling \eqref{identityPiu} and using the above inequality, we obtain
\begin{align*}
\fanshu[0,K]{\pmb{u}-\pmb{\Pi}_{U}\pmb{u}}&\le \fanshu[0,K]{\pmb{u}-\pmb{\Pi}_{N}\pmb{u}}+\fanshu[0,K]{\pmb{\Pi}_{B}(\pmb{u}-\pmb{\Pi}_{N}\pmb{u})}\\
&\lesssim\fanshu[0,K]{\pmb{u}-\pmb{\Pi}_{N}\pmb{u}}+h_K^{3/2}\fanshu[0,F]{\curl(\pmb{u}-\pmb{\Pi}_{N}\pmb{u})}.
\end{align*}
As a consequence, by the approximation results of $\pmb{\Pi}_{N}$\cite{monk2003finite}, we get \eqref{errorL2PiU}, which ends our proof.
\end{proof}

\section{Application to the quad-curl problem}\label{sectionApplication}
In this section, we use the prosed nonconforming finite elements to solve the following quad-curl problem
\begin{equation}\label{quadcurlproblem}
\begin{aligned}
\nabla\times\nabla\times\nabla\times\nabla\times\pmb{u}&=\pmb{f}\ \text{in }\Omega,\\
\nabla\cdot\pmb{u}&=0\ \text{in }\Omega,\\
\pmb{u}\times\pmb{n}&=0\ \text{on }\partial\Omega,\\
(\nabla\times\pmb{u})\times\pmb{n}&=0\ \text{on }\partial\Omega,
\end{aligned}
\end{equation}
where $\pmb{f}\in \pmb{H}(\mathrm{div}^0;\Omega)$.

According to \cite{huang2022nonconforing}, the quad-curl problem \eqref{quadcurlproblem} is equivalent to
\begin{equation}\label{curlLapcurlproblem}
\begin{aligned}
-\nabla\times\Delta(\nabla\times\pmb{u})&=\pmb{f}\ \text{in }\Omega,\\
\nabla\cdot\pmb{u}&=0\ \text{in }\Omega,\\
\pmb{u}\times\pmb{n}&=0\ \text{on }\partial\Omega,\\
(\nabla\times\pmb{u})&=0\ \text{on }\partial\Omega,
\end{aligned}
\end{equation}
The mixed formulation of \eqref{curlLapcurlproblem} reads: find $(\pmb{u},p)\in \pmb{H}_0(\mathrm{grad}\ \mathrm{curl};\Omega)\times H_0^1(\Omega)$ such that
\begin{equation}\label{variationalcurlLapcurlproblem}
\begin{aligned}
a(\pmb{u},\pmb{v})+b(\pmb{v},p)&=(\pmb{f},\pmb{v})\quad\forall\pmb{v}\in \pmb{H}_0(\mathrm{grad}\ \mathrm{curl};\Omega),\\
b(\pmb{u},q)&=0\qquad\qquad\quad \forall q\in H_0^1(\Omega),
\end{aligned}
\end{equation}
where
\begin{align*}
a(\pmb{u},\pmb{v})&=(\gcurl\pmb{u},\gcurl\pmb{v}),\\
b(\pmb{v},{q})&=(\pmb{v},\nabla q).
\end{align*}

Testing $\pmb{v}=\nabla q,\ \forall q\in H_0^1(\Omega)$ in the first equation of \eqref{variationalcurlLapcurlproblem}, we obtain $p=0$ from the fact $\divv\pmb{f}=0$. Therefore we have
\begin{align}
a(\pmb{u},\pmb{v})=(\pmb{f},\pmb{v})\quad\forall\pmb{v}\in \pmb{H}_0(\mathrm{grad}\ \mathrm{curl};\Omega).
\end{align}

\subsection{Mixed finite element}

Based on \eqref{variationalcurlLapcurlproblem}, we propose the finite element method for \eqref{curlLapcurlproblem} as follows: find $(\pmb{u}_h,p_h)\in\pmb{U}_{h,0}\times W_{h,0}$ such that
\begin{align}
\begin{aligned}\label{discretecurlLapcurlproblem}
a_h(\pmb{u}_h,\pmb{v}_h)+b(\pmb{v}_h, p_h)&=(\pmb{f},\pmb{v}_h),\ \forall\pmb{v}_h\in\pmb{U}_{h,0},\\
b(\pmb{u}_h,q_h)&=0,\qquad \forall q_h\in W_{h,0},
\end{aligned}
\end{align}
where $a_h(\pmb{u}_h,\pmb{v}_h)=(\nabla_h\curl\pmb{u}_h,\nabla_h\curl\pmb{v}_h),$ and $\nabla_h$ is the elementwise version of $\nabla$ with respect to $\mesh$.
We define the following norm for $\pmb{v}_h\in\pmb{U}_h$
\begin{align}
\norma{\pmb{v}_h}^2=\fanshu[0]{\pmb{v}_h}^2+\fanshu[0]{\curl\pmb{v}_h}^2+a_h(\pmb{v}_h,\pmb{v}_h).
\end{align}

Now we discuss the well-posedness of the mixed method \eqref{discretecurlLapcurlproblem}. To this end, we introduce the following discrete Friedrichs inequality, which can be proved in a similar way as in
\cite[Lemma 7.20]{monk2003finite}.
\begin{lemma}\label{discreteFri}
Assume that $\Omega$ is a bounded Lipschitz polyhedron in $\mathbb{R}^3$. Let $\pmb{U}_{h,0}$ be defined as \eqref{definitionUh0}.  Then for $h$ small enough, we have
\begin{align*}
\fanshu[0]{\pmb{v}_h}\lesssim\fanshu[0]{\curl\pmb{v}_h},\ \forall\pmb{v}_h\in\widetilde{\pmb{Z}}_{h},
\end{align*}
where
\begin{align*}
\widetilde{\pmb{Z}}_h = \{\pmb{v}_h\in\pmb{U}_{h}\cap\pmb{H}_0(\mathrm{curl};\Omega):b(\pmb{v}_h,q_h)=0,\ \forall q_h\in W_{h,0} \}.
\end{align*}
\end{lemma}
\begin{remark}
Set
\begin{align*}
\pmb{Z}_h = \{\pmb{v}_h\in\pmb{U}_{h,0}:b(\pmb{v}_h,q_h)=0,\ \forall q_h\in W_{h,0} \}.
\end{align*}
Due to $\pmb{Z}_h\subset\widetilde{\pmb{Z}_h}$, we also have
\begin{align*}
\fanshu[0]{\pmb{v}_h}\lesssim\fanshu[0]{\curl\pmb{v}_h},\ \forall\pmb{v}_h\in\pmb{Z}_{h}.
\end{align*}
\end{remark}

We also introduce the following discrete Poincar\'{e}--Friedrichs for $\pmb{v}_h\in\pmb{V}_{h,0}$, which is a direct result of \cite{Brenner2003Poincare}.
\begin{lemma}\label{discretePoincareFri}
Let $\pmb{V}_{h,0}$ be defined as \eqref{definitionVh0}.  Then we have
\begin{align*}
\fanshu[0]{\pmb{v}_h}\lesssim\fanshu[0]{\nabla_h\pmb{v}_h},\ \forall\pmb{v}_h\in\pmb{V}_{h,0}.
\end{align*}
\end{lemma}
\begin{lemma}
We have the discrete stability
\begin{align}\label{discreteStability}
&\norma{\widetilde{\pmb{u}}_h}+|\widetilde{p}_h|_1\lesssim\sup_{(\pmb{v}_h,q_h)\in\pmb{U}_{h,0}\times W_{h,0}}\frac{a_h(\widetilde{\pmb{u}}_h,\pmb{v}_h)+b(\pmb{v}_h,\widetilde{p}_h)+b(\widetilde{\pmb{u}}_h,q_h)}{\norma{\pmb{v}_h}+|q_h|_1}
\end{align}
for any $(\widetilde{\pmb{u}}_h,\widetilde{p}_h)\in\pmb{U}_{h,0}\times W_{h,0}$.
\end{lemma}
\begin{proof}
For any $\pmb{v}_h\in\pmb{Z}_h$, noting $\curl\pmb{v}_h\in\pmb{V}_{h,0}$, then using Lemma \ref{discretePoincareFri} and Lemma \ref{discreteFri}, we derive the coercivity
\begin{align}
\norma{\pmb{v}_h}^2\lesssim \fanshu[0]{\curl\pmb{v}_h}^2+\fanshu[0]{\nabla_h\curl\pmb{v}_h}^2\lesssim a_h(\pmb{v}_h,\pmb{v}_h).
\end{align}
Since $\nabla q_h\in \pmb{U}_{h,0},\ \forall q_h\in W_{h,0}$, we have the inf-sup condition
\begin{align}
\sup_{\pmb{0}\neq\pmb{v}_h\in\pmb{U}_h}\frac{b(\pmb{v}_h,q_h)}{\norma{\pmb{v}_h}}\ge \frac{b(\nabla q_h,q_h)}{\norma{\nabla q_h}} = |q_h|_1 \quad\forall q_h\in W_{h,0}.
\end{align}
As a result, the discrete stability follows from the Babu\v{s}ka-Brezzi theory \cite{boffi2013mixed}.
\end{proof}

From \eqref{discreteStability}, we can easily get that the mixed method \eqref{discretecurlLapcurlproblem} is well-posed. As the continuous case, testing $\pmb{v}_h=\nabla q_h,\ \forall q_h\in W_{h,0}$ in the first equation of \eqref{discretecurlLapcurlproblem}, we deduce $p_h=0$ from the fact $\divv\pmb{f}=0$. Therefore, the solution $\pmb{u}_h\in\pmb{U}_{h,0}$ satisfies
\begin{align}
a_h(\pmb{u}_h,\pmb{v}_h)=(\pmb{f},\pmb{v}_h),\quad\forall\pmb{v}_h\in\pmb{U}_{h,0}.\label{equationahfh}
\end{align}

\subsection{Error analysis}
We now turn to the error analysis for the mixed finite element method \eqref{discretecurlLapcurlproblem}. To begin with, we show the consistency error estimates for the proposed nonconforming method.
\begin{lemma}\label{Lemmaconsistencyerr}
Let $\pmb{u}\in\pmb{H}(\mathrm{grad}~\mathrm{curl};\Omega)$ be the solution of \eqref{curlLapcurlproblem}, and $s$ the integer satisfying $2\le s\le k+1$. If $\pmb{u},\curl\pmb{u}\in \pmb{H}^{s}(\Omega)$, then we have for any $\pmb{v}_h\in\pmb{U}_{h,0}$ that
\begin{align}
a_h(\pmb{u},\pmb{v}_h)-(\pmb{f},\pmb{v}_h)\lesssim h^{s-1}|\curl\pmb{u}|_{s}\norma{\pmb{v}_h}.
\end{align}
\end{lemma}
\begin{proof}
Integrating by parts, we obtain
\begin{align*}
a_h(\pmb{u},\pmb{v}_h)-(\pmb{f},\pmb{v}_h)=\sum_{K\in\mesh}\langle\partial_n\curl\pmb{u},\curl\pmb{v}_h\rangle_{\partial K}.
\end{align*}
By Lemma \ref{Lemmajumpcurl}, we derive
\begin{align*}
&\sum_{K\in\mesh}\langle\partial_n\curl\pmb{u},\curl\pmb{v}_h\rangle_{\partial K}\\
&=\sum_{F\in\face}\langle\partial_n\curl\pmb{u}-\pmb{Q}_F^{k-1}\partial_n\curl\pmb{u},\jump{\curl\pmb{v}_h}-\pmb{Q}_F^{0}\jump{\curl\pmb{v}_h}\rangle_F\\
&\lesssim h^{s-1}|\curl\pmb{u}|_s\norma{\pmb{v}_h},
\end{align*}
where $\pmb{Q}_F^{k-1}$ is the $L^2$-orthogonal projection from $\pmb{L}^2(F)$ to $\pmb{P}^{k-1}(F)$. We have thus proved the lemma.
\end{proof}

With the help of preceding lemmas we can now prove the \textit{a priori} error estimate.
\begin{theorem}\label{theoremerrorcurlcurl}
Let $\pmb{u}\in\pmb{H}(\mathrm{grad}~\mathrm{curl};\Omega)$ be the solution of \eqref{curlLapcurlproblem}, and $\pmb{u}_h\in\pmb{U}_h$ the solution of \eqref{discretecurlLapcurlproblem}. If $\pmb{u},\curl\pmb{u}\in \pmb{H}^{s}(\Omega)$ with $2\le s\le k+1$, then we have
\begin{align}
\norma{\pmb{u}-\pmb{u}_h}\lesssim h^{s-1}(|\pmb{u}|_{s}+|\curl\pmb{u}|_{s}).
\end{align}
\end{theorem}
\begin{proof}
Applying \eqref{discreteStability} with $\widetilde{\pmb{u}}_h=\pmb{\Pi}_{U}\pmb{u}-\pmb{u}_h$ and $\widetilde{p}_h=0$, recalling \eqref{equationahfh} and $\pmb{u}_h\in\pmb{Z}_h$, we have
\begin{align*}
\norma{\pmb{\Pi}_{U}\pmb{u}-\pmb{u}_h}&\lesssim\sup_{(\pmb{v}_h,q_h)\in\pmb{U}_{h,0}\times W_{h,0}}\frac{a_h(\pmb{\Pi}_{U}\pmb{u}-\pmb{u}_h,\pmb{v}_h)+b(\pmb{\Pi}_{U}\pmb{u}-\pmb{u}_h,q_h)}{\norma{\pmb{v}_h}+|q_h|_1}\\
&=\sup_{(\pmb{v}_h,q_h)\in\pmb{U}_{h,0}\times W_{h,0}}\frac{a_h(\pmb{\Pi}_{U}\pmb{u},\pmb{v}_h)-(\pmb{f},\pmb{v}_h)+b(\pmb{\Pi}_{U}\pmb{u}-\pmb{u},q_h)}{\norma{\pmb{v}_h}+|q_h|_1}\\
&\lesssim\fanshu[0]{\pmb{\Pi}_{U}\pmb{u}-\pmb{u}}+\fanshu[0]{\nabla_h\curl(\pmb{\Pi}_{U}\pmb{u}-\pmb{u})}+\sup_{\pmb{v}_h\in\pmb{U}_{h,0}}\frac{a_h(\pmb{u},\pmb{v}_h)-(\pmb{f},\pmb{v}_h)}{\norma{\pmb{v}_h}}.
\end{align*}
Using Theorem \ref{theoreminterpo}, Lemma \ref{Lemmaconsistencyerr} and the triangle inequality, we can now conclude the desired result.
\end{proof}

Next we present an optimal estimate for the error $\fanshu[0]{\curl(\pmb{u}-\pmb{u}_h)}$. For this purpose, we consider the dual problem
\begin{equation}\label{dualproble}
\begin{aligned}
-\nabla\times\Delta(\nabla\times\widetilde{\pmb{u}})&=\curlc(\pmb{u}-\pmb{u}_h)\ \text{in }\Omega,\\
\nabla\cdot\widetilde{\pmb{u}}&=0\qquad\qquad\qquad \text{ in }\Omega,\\
\widetilde{\pmb{u}}\times\pmb{n}&=0\qquad\qquad\qquad \text{ on }\partial\Omega,\\
(\nabla\times\widetilde{\pmb{u}})&=0\qquad\qquad\qquad \text{ on }\partial\Omega,
\end{aligned}
\end{equation}
where $\widetilde{\pmb{u}}\in\pmb{H}_0(\mathrm{grad}~\mathrm{curl})$. It is worth pointing out that the first equation of \eqref{dualproble} holds in the sense of $\pmb{H}^{-1}(\mathrm{div};\Omega)$, where $\pmb{H}^{-1}(\mathrm{div};\Omega)$ is the dual space of $\pmb{H}_0(\mathrm{curl};\Omega)$\cite{chen2018decoupling}. According to \cite[Lemma A.1]{huang2022nonconforing}, if $\Omega$ is convex, then the following regularity holds
\begin{align}\label{regularitydual}
\fanshu[1]{\widetilde{\pmb{u}}}+\fanshu[2]{\curl\widetilde{\pmb{u}}}\lesssim\fanshu[-1]{\curlc(\pmb{u}-\pmb{u}_h)}\lesssim\fanshu[0]{\curl(\pmb{u}-\pmb{u}_h)}.
\end{align}
\begin{theorem}\label{theoremErrHcurl}
Under the conditions of Theorem \ref{theoremerrorcurlcurl}, we further assume that the regularity \eqref{regularitydual} holds. Then we have
\begin{align}
\fanshu[0]{\curl(\pmb{u}-\pmb{u}_h)}\lesssim h^s(|\pmb{u}|_{s}+|\curl\pmb{u}|_{s}).\label{errorHcurl}
\end{align}
\end{theorem}
\begin{proof}
Multiplying both sides of the first equation of \eqref{dualproble} by $\pmb{u}-\pmb{u}_h$ and integrating by parts, we get
\begin{align*}
\fanshu[0]{\curl(\pmb{u}-\pmb{u}_h)}^2&=a_h(\widetilde{\pmb{u}}-\pmb{\Pi}_{U}\widetilde{\pmb{u}},\pmb{u}-\pmb{u}_h)+a_h(\pmb{\Pi}_{U}\widetilde{\pmb{u}},\pmb{u}-\pmb{u}_h)\\
&-\sum_{K\in\mesh}\langle\partial_n\curl\widetilde{\pmb{u}},\curl(\pmb{u}-\pmb{u}_h)\rangle_{\partial K},
\end{align*}
Integrating by parts for the second term of the above equation and using the first equation of \eqref{discretecurlLapcurlproblem}, we obtain
\begin{align}\label{equation2}
\fanshu[0]{\curl(\pmb{u}-\pmb{u}_h)}^2&=a_h(\widetilde{\pmb{u}}-\pmb{\Pi}_{U}\widetilde{\pmb{u}},\pmb{u}-\pmb{u}_h)-\sum_{K\in\mesh}\langle\partial_n\curl\widetilde{\pmb{u}},\curl(\pmb{u}-\pmb{u}_h)\rangle_{\partial K}\\
&+\sum_{K\in\mesh}\langle\curl(\pmb{\Pi}_{U}\widetilde{\pmb{u}}),\partial_n\curl\pmb{u}\rangle_{\partial K}.
\end{align}
It follows from Theorem \ref{theoreminterpo}, Theorem \ref{theoremerrorcurlcurl} and \eqref{regularitydual} that
\begin{align}\label{estimate1}
a_h(\widetilde{\pmb{u}}-\pmb{\Pi}_{U}\widetilde{\pmb{u}},\pmb{u}-\pmb{u}_h)&\lesssim\fanshu{\nabla_h\curl(\widetilde{\pmb{u}}-\pmb{\Pi}_{U}\widetilde{\pmb{u}})}\fanshu{\nabla_h\curl(\pmb{u}-\pmb{u}_h)}\notag\\
&\lesssim h^s|\curl\widetilde{\pmb{u}}|_2(|\pmb{u}|_{s}+|\curl\pmb{u}|_{s})\notag\\
&\lesssim h^s\fanshu[0]{\curl(\pmb{u}-\pmb{u}_h)}(|\pmb{u}|_{s}+|\curl\pmb{u}|_{s}).
\end{align}
Thanks to Lemma \ref{Lemmajumpcurl}, Theorem \ref{theoreminterpo} and Theorem \ref{theoremerrorcurlcurl}, we have
\begin{align}\label{estimate2}
\sum_{K\in\mesh}\langle\partial_n\curl\widetilde{\pmb{u}},\curl(\pmb{u}-\pmb{u}_h)\rangle_{\partial K}&=\sum_{F\in\face}\langle\partial_n\curl\widetilde{\pmb{u}}-\pmb{Q}_F^{k-1}\partial_n\curl\widetilde{\pmb{u}},\jump{\curl(\pmb{u}-\pmb{u}_h)}\rangle_{F}\notag\\
&\lesssim h^s|\curl\widetilde{\pmb{u}}|_2(|\pmb{u}|_{s}+|\curl\pmb{u}|_{s})\notag\\
&\lesssim h^s\fanshu[0]{\curl(\pmb{u}-\pmb{u}_h)}(|\pmb{u}|_{s}+|\curl\pmb{u}|_{s}).
\end{align}
Similarly we derive
\begin{align}\label{estimate3}
\sum_{K\in\mesh}\langle\curl(\pmb{\Pi}_{U}\widetilde{\pmb{u}}),\partial_n\curl\pmb{u}\rangle_{\partial K}&=\sum_{F\in \face}\langle\jump{\curl(\pmb{\Pi}_{U}\widetilde{\pmb{u}})-\widetilde{\pmb{u}}},\partial_n\curl\pmb{u}-\pmb{Q}_F^{k-1}\partial_n\curl\pmb{u}\rangle_F\notag\\
&\lesssim h^s|\curl\widetilde{\pmb{u}}|_2|\curl\pmb{u}|_{s}\notag\\
&\lesssim h^s\fanshu[0]{\curl(\pmb{u}-\pmb{u}_h)}|\curl\pmb{u}|_{s}.
\end{align}
Combing \eqref{equation2}-\eqref{estimate3}, we arrive at \eqref{errorHcurl}.
\end{proof}

We end this section by considering an estimate for $\fanshu[0]{\pmb{u}-\pmb{u}_h}$.
\begin{theorem}\label{theoremErrL2}
Under the conditions of Theorem \ref{theoremErrHcurl}, we have
\begin{align}
\fanshu[0]{\pmb{u}-\pmb{u}_h}\lesssim h^s(|\pmb{u}|_{s}+|\curl\pmb{u}|_{s}).\label{errL2}
\end{align}
\end{theorem}
\begin{proof}
We shall adopt the same procedure as in the proof of \cite[Theorem 6]{Sun2016MixedFemQuad}. Let $\overline{\pmb{u}}_h\in\pmb{U}_{h}\cap \pmb{H}_0(\mathrm{curl};\Omega)$ is the solution of the following finite element method with $\pmb{g}=\curlc\pmb{u}$
\begin{equation}\label{discreteMaxwell}
\begin{aligned}
(\curl\overline{\pmb{u}}_h,\curl\pmb{v}_h)+(\nabla\pmb{v}_h,\overline{p}_h)&=(\pmb{g},\pmb{v}_h)\quad\forall\pmb{v}_h\in\pmb{U}_{h}\cap \pmb{H}_0(\mathrm{curl};\Omega),\\
(\nabla\overline{\pmb{u}}_h,q_h)&=0\qquad\quad\ \forall\pmb{q}_h\in\pmb{W}_{h,0}.
\end{aligned}
\end{equation}
Since $\nabla q_h\in \pmb{U}_{h}\cap\pmb{H}_0(\mathrm{curl};\Omega),\ \forall q_h\in W_{h,0}$, we have the inf-sup condition
\begin{align*}
\sup_{\pmb{0}\neq\pmb{v}_h\in\pmb{U}_h}\frac{b(\pmb{v}_h,q_h)}{\fanshu[H(\mathrm{curl};\Omega)]{\pmb{v}_h}}\ge \frac{b(\nabla q_h,q_h)}{\fanshu[H(\mathrm{curl};\Omega)]{\nabla q_h}} = |q_h|_1 \quad\forall q_h\in W_{h,0}.
\end{align*}
Then combing the above inf-sup condition and Lemma \ref{discreteFri}, we conclude that \eqref{discreteMaxwell} is well-posed. Since $\pmb{U}_{h}\subset\pmb{H}(\mathrm{curl};\Omega)$, \eqref{discreteMaxwell} is a conforming method for the Maxwell's equations whose exact solution is $\pmb{u}$. Therefore by standard finite element framework and Theorem \ref{theoreminterpo}, we have the following error estimate
\begin{align}
\fanshu[H(\mathrm{curl};\Omega)]{\pmb{u}-\overline{\pmb{u}}_h}\lesssim h^s(|\pmb{u}|_{s}+|\curl\pmb{u}|_{s}).\label{errMaxwell}
\end{align}
Then using the triangle inequality and Lemma \ref{discreteFri} we derive
\begin{align*}
\fanshu[0]{\pmb{u}-\pmb{u}_h}&\le \fanshu[0]{\pmb{u}-\overline{\pmb{u}}_h}+\fanshu[0]{\overline{\pmb{u}}_h-\pmb{u}_h}\\
&\lesssim\fanshu[0]{\pmb{u}-\overline{\pmb{u}}_h}+\fanshu[0]{\curl(\overline{\pmb{u}}_h-\pmb{u}_h)}\\
&\lesssim\fanshu[H(\mathrm{curl};\Omega)]{\pmb{u}-\overline{\pmb{u}}_h}+\fanshu[0]{\curl(\pmb{u}-\pmb{u}_h)}.
\end{align*}
Finally \eqref{errL2} follows from the above estimate, \eqref{errMaxwell} and \eqref{errorHcurl}.
\end{proof}

\section{Numerical experiments}\label{sectionnumerical}
In this section, we report some numerical results of our new nonconforming methods for the quad-curl equation \eqref{curlLapcurlproblem}. Our algorithms are implemented by the iFEM package \cite{chen2008ifem}.
We demonstrate the following relative errors:
\begin{align*}
E_{L^2}&=\fanshu[0]{\pmb{u}-\pmb{u}_h}/\fanshu[0]{\pmb{u}},\\
E_{\mathrm{curl}}&=\fanshu[0]{\nabla\times(\pmb{u}-\pmb{u}_h)}/\fanshu[0]{\curl\pmb{u}},\\
E_{\mathrm{grad}~\mathrm{curl}}&=\fanshu[0]{\nabla_h\curl{(\pmb{u}-\pmb{u}_h)}}/\fanshu[0]{\gcurl\pmb{u}}.
\end{align*}
\subsection{Example 1}
We consider the quad-curl problem \eqref{curlLapcurlproblem} on the unit cube $\Omega=(0,1)^3$. The right-hand side $\pmb{f}$ is determined by the exact solution
\begin{align}
\pmb{u}=\left(
          \begin{array}{c}
            \sin(\pi x)^3\sin(\pi y)^2\sin(\pi z)^2\cos(\pi y)\cos(\pi z) \\
            \sin(\pi y)^3\sin(\pi z)^2\sin(\pi x)^2\cos(\pi z)\cos(\pi x) \\
          -2\sin(\pi z)^3\sin(\pi x)^2\sin(\pi y)^2\cos(\pi x)\cos(\pi y) \\
          \end{array}
        \right).
\end{align}
The mesh $\mesh$ is obtained by dividing the unit cube into $N^3$ small cubes and then partition each small cube into $6$ tetrahedra. The convergence results for $k=1$ are presented in Table \ref{tableexample1}. We see that all three asymptotic convergence rates are optimal, which confirms the theoretical analysis in Theorems \ref{theoremerrorcurlcurl}, \ref{theoremErrHcurl} and \ref{theoremErrL2}.
\begin{table}[!htbp]\centering
\caption{Example 1: Numerical results by the lowest-order nonconforming grad-curl element ($k=1$)\label{tableexample1}}
\scalebox{1}[1]{
\begin{tabular}{*{8}{c@{\extracolsep{5pt}}@{\extracolsep{5pt}}}}\hline
$N$&$h$ &$E_{L^2}$&  rate & $E_{\rm{curl}}$&  rate  & $E_{\mathrm{grad}~\mathrm{curl}}$ &  rate \\\hline
8 & 0.2165& 1.567e-01&      & 1.797e-01&      & 4.552e-01&      \\
10 & 0.1732& 1.071e-01& 1.71& 1.241e-01& 1.66& 3.759e-01& 0.86\\
12 & 0.1443& 7.723e-02& 1.79& 9.012e-02& 1.76& 3.189e-01& 0.90\\
14 & 0.1237& 5.812e-02& 1.84& 6.811e-02& 1.82& 2.765e-01& 0.93\\
16 & 0.1083& 4.523e-02& 1.88& 5.315e-02& 1.86& 2.437e-01& 0.94\\
18 & 0.0962& 3.615e-02& 1.90& 4.257e-02& 1.89& 2.178e-01& 0.96\\
20 & 0.0866& 2.952e-02& 1.92& 3.482e-02& 1.91& 1.968e-01& 0.96\\
22 & 0.0787& 2.455e-02& 1.93& 2.899e-02& 1.92& 1.794e-01& 0.97\\
\hline
\end{tabular}}
\end{table}

\subsection{Example 2}
We now consider nonhomogeneous boundary condition on the unit cube $\Omega=(0,1)^3$, with the exact solution
\begin{align}
\pmb{u}=\left(
          \begin{array}{c}
            \sin(y)\sin(z) \\
            \sin(z)\sin(x) \\
            \sin(x)\sin(y) \\
          \end{array}
        \right).
\end{align}
It is easy to check that
\begin{align}
\pmb{u}\times\pmb{n}|_{\partial\Omega}\neq\pmb{0},\ \curl\pmb{u}|_{\partial\Omega}\neq\pmb{0}.
\end{align}

We use the same meshes as the previous example. Numerical results are presented in Table \ref{tableexample2}. We observe that all errors converge with optimal rates. Although our theory (Theorems 4.1-4.3) does not cover the case of nonhomogeneous boundary condition, our method still works for this case.
\begin{table}[!htbp]\centering
\caption{Example 2: Numerical results by the lowest-order nonconforming grad-curl element ($k=1$)\label{tableexample2}}
\scalebox{1}[1]{
\begin{tabular}{*{8}{c@{\extracolsep{5pt}}@{\extracolsep{5pt}}}}\hline
$N$&$h$ &$E_{L^2}$&  rate & $E_{\rm{curl}}$&  rate  & $E_{\mathrm{grad}~\mathrm{curl}}$ &  rate \\\hline
8 & 0.2165& 4.164e-03&      & 7.521e-03&      & 1.079e-01&      \\
10 & 0.1732& 2.665e-03& 2.00& 4.819e-03& 1.99& 8.635e-02& 1.00\\
12 & 0.1443& 1.851e-03& 2.00& 3.349e-03& 2.00& 7.198e-02& 1.00\\
14 & 0.1237& 1.360e-03& 2.00& 2.461e-03& 2.00& 6.172e-02& 1.00\\
16 & 0.1083& 1.041e-03& 2.00& 1.885e-03& 2.00& 5.401e-02& 1.00\\
18 & 0.0962& 8.225e-04& 2.00& 1.490e-03& 2.00& 4.802e-02& 1.00\\
20 & 0.0866& 6.662e-04& 2.00& 1.207e-03& 2.00& 4.323e-02& 1.00\\
22 & 0.0787& 5.506e-04& 2.00& 9.977e-04& 2.00& 3.930e-02& 1.00\\
\hline
\end{tabular}}
\end{table}
\section*{Acknowledgments}
The authors would like to thank anonymous referees for their valuable comments. This work is supported in part by the National Natural Science Foundation of China grants NSFC 11871092, 12131005, and NSAF U1930402.
\bibliographystyle{siam}
\bibliography{QuadCurl}
\end{document}